\documentclass[12pt,a4paper]{amsart}
\usepackage{a4wide}
\usepackage{amsmath, amsthm, amssymb}
\usepackage{color, multirow}
\usepackage{enumerate}

\newtheorem{theorem}{Theorem}[section]

\newtheorem{proposition}[theorem]{Proposition}

\theoremstyle{remark}
\newtheorem{remark}[theorem]{Remark}

\newtheorem*{claim*}{Claim}

\newcommand{\g}[1]{\ensuremath{\mathfrak{#1}}}

\newcommand{\FM}{\ensuremath{\mathcal{F}}}

\newcommand{\bR}{\ensuremath{\mathbb{R}}}
\newcommand{\bC}{\ensuremath{\mathbb{C}}}

\newcommand{\bV}{\ensuremath{\mathbb{V}}}
\newcommand{\bW}{\ensuremath{\mathbb{W}}}

\newcommand{\ads}{\ensuremath{AdS^{n+1}}}
\newcommand{\adss}{\ensuremath{AdS^{2n+1}}}
\newcommand{\chn}{\ensuremath{\bC H^n}}
\newcommand{\chk}{\ensuremath{\bC H^k}}
\newcommand{\trsp}{\ensuremath{^{\intercal}}}
\newcommand{\ii}{\ensuremath{\mathrm{i}}}

\DeclareMathOperator{\tr}{tr}
\DeclareMathOperator{\re}{Re}

\DeclareMathOperator{\ad}{ad}
\DeclareMathOperator{\Exp}{Exp}

\DeclareMathOperator{\spann}{span}
\DeclareMathOperator{\diag}{diag}

\DeclareMathOperator{\B}{B}

\begin{document}
\title[Cohomogeneity one anti de Sitter spacetime]{Cohomogeneity one actions\\ on anti de Sitter spacetimes}

\author[J.\ C.\ D\'{\i}az-Ramos]{J.C. D\'{\i}az-Ramos}
\author[S.M.B.\ Kashani]{S.M.B. Kashani}
\author[M.J.\ Vanaei]{M.J. Vanaei}

\address{Department of Mathematics, Universidade de Santiago de Compostela, Spain}
\address{Dept.\ of Pure Math., School of Math.\ sciences, Tarbiat Modares University, Tehran, Iran, P.O.~Box 14115-134}
\address{Dept.\ of Pure Math., School of Math.\ sciences, Tarbiat Modares University, Tehran, Iran, P.O.~Box 14115-134}

\email{josecarlos.diaz@usc.es}
\email{kashanim@modares.ac.ir}
\email{javad.vanaei@modares.ac.ir}

\thanks{The second and the third authors have been supported by the Iranian presidential office via grant no.\ 88001210. The first author has been supported by projects EM2014/009, GRC2013-045 and MTM2013-41335-P with FEDER funds (Spain)}

\begin{abstract}
In this paper we classify, up to orbit equivalence, cohomogeneity one actions of connected closed Lie subgroups of $U(1,n)$ on the $(2n+1)$-dimensional anti de Sitter spacetime $\adss$. We also give some new examples of nonproper cohomogeneity one actions on $\ads$ and determine parabolic Lie subgroups of $SO(2,n)$ and their orbits in $\ads$.
\end{abstract}

\date{\today}



\maketitle


\section{Introduction}

Cohomogeneity one actions have successfully been used in Riemannian geometry to construct examples of manifolds with certain geometric properties. See for example~\cite{GrZi00} or~\cite{Wi06} for a remarkable relation between cohomogeneity one actions and Riemannian manifolds of positive curvature. Other uses of cohomogeneity one actions can be found to build examples of Einstein metrics or Ricci solitons~\cite{DaWa11}. They have also been used to construct examples of submanifolds with symmetries in~\cite{GoGu00}. In general, many geometric conditions translate into a difficult PDE on a manifold; if the manifold is of cohomogeneity one and the PDE behaves correctly with respect to this structure, then this PDE becomes an ODE which one might be capable of solving. A recent example of this procedure can be found in~\cite{DiDoVi}, where the mere fact of having existence and uniqueness results for ordinary differential equations allows, at least theoretically, to solve a problem.

Our motivation for studying cohomogeneity one actions comes from a different perspective in this paper. We are interested in the classification and study of cohomogeneity one actions on a given manifold. Cohomogeneity one actions on Euclidean spaces were classified by Segre~\cite{Se38} in his study of isoparametric hypersurfaces. Kollross gave a classification of cohomogeneity one actions on compact irreducible symmetric spaces in~\cite{Ko02}. For irreducible symmetric spaces of noncompact type some important progress has been done by Berndt and Tamaru~\cite{BeTa13}. The classification of cohomogeneity one actions on complex hyperbolic spaces, which was also achieved by Berndt and Tamaru in~\cite{BeTa07}, is particularly relevant for this paper. It can also be obtained as a corollary to the more general classification of isoparametric hypersurfaces in complex hyperbolic spaces~\cite{DiDoSa}.

The study of isometric actions on Lorentzian manifolds is not as well developed as in Riemannian manifolds, although some interesting results are obtained by Adams and Stuck in~\cite{AdSt97}. One important question that arises in this setting is whether nonproper actions should be investigated. On the one hand, the fact that a Lie group action is not proper makes the study much more complicated. For example, in~\cite{BeDiVa} we have shown that there might exist cohomogeneity zero actions on a Lorentzian manifold of constant curvature, which are nevertheless not transitive. Other curious phenomena such as the existence of non-closed orbits may occur. On the other hand, there are many interesting non proper actions that seem worthwhile to study. One such example is the action of $O(1,n)$ on the Minkowski space $\mathbb{L}^{n+1}$, whose orbits are real hyperbolic spaces, de Sitter spacetimes and light cones, all of which have geometric or physical meaning.

In this paper we study cohomogeneity one actions on anti de Sitter spacetimes, as a continuation of a previous study in~\cite{VKS16}. In the first part of the paper we assume that our action is proper and study actions on odd dimensional anti de Sitter spacetimes that are related to cohomogeneity one actions on complex hyperbolic spaces via the Hopf map. We also give a description of their orbit spaces, and how the orbits of these actions can be obtained geometrically. This part is basically an application of the classification of cohomogeneity one actions on complex hyperbolic spaces~\cite{BeTa07}, and of transitive actions on complex hyperbolic spaces~\cite{CaGaSw09}.

In the second part of this paper we give new examples of cohomogeneity one and cohomogeneity zero actions on anti de Sitter spacetimes of any dimension. More specifically, we consider the Iwasawa decomposition of $SO^0(2,n)$ and parabolic subgroups of $SO^0(2,n)$. Most of the examples provided are not proper. Hence, there are cohomogeneity zero actions that are not transitive, and cohomogeneity one actions with infinitely many singular orbits. In these cases there are also orbits that are not closed. This contrasts with the Riemannian setting, where these natural decompositions of isometry groups give rise to well-behaved proper isometric actions.

The paper is organized as follows. In Section~\ref{sec:preliminaries} we present the notions and notations that are used in this paper. In Section~\ref{sec: C1-LieSub-SU(1,n)} we study subgroups of $U(1,n)$ acting with cohomogeneity one on odd dimensional anti de Sitter spacetimes. First we introduce some results and notation in Subsection~\ref{subsec:chn}, and then we address our study in Subsection~\ref{subsec:u1n}, to prove one of our main results: Theorem~\ref{Th.-C1-Liesub-SU(1,n)}. In Section~\ref{sec:parabolicsubgroups} we give new examples of cohomogeneity zero and cohomogeneity one actions on anti de Sitter spacetimes. We give in Subsection~\ref{Iwa-SO} the explicit root space decomposition of the Lie algebra $\g{so}(2,n)$. In subsections~\ref{sec:N} and~\ref{subsec: Parab-Subg-SO(2,n)} we take advantage of these results to determine the Iwasawa decomposition and the parabolic subalgebras of $\g{so}(2,n)$. The main contributions in this section are propositions~\ref{N-action},~\ref{th:parab-action} and other interesting examples of isometric actions together with the study of their orbits.


\section{Preliminaries}\label{sec:preliminaries}

We start with some basic definitions and results related to anti de Sitter spacetimes and cohomogeneity one actions.

\subsection{Anti de Sitter spacetimes}\label{sec:antiDeSitter}\hfill

Throughout the paper, $\mathbb{R}^{\nu,n-\nu}$, $0\leq \nu \leq n$, stands for the pseudo-Euclidean vector space $\mathbb{R}^n$ endowed with the standard scalar product of signature $(\nu,n- \nu)$. For $r>0$ we define the quadric hypersurfaces
\begin{align*}
S^n_{\nu}(r)&{}= \{x \in \mathbb{R}^{\nu, n-\nu+1} : \langle u, u\rangle =r^2\},&
H^n_{\nu}(r)&{}= \{x \in \mathbb{R}^{\nu+1,n-\nu} : \langle u, u\rangle =-r^2\},
\end{align*}
called, respectively, the pseudo-sphere and the pseudo-hyperbolic space of index $\nu$ and radius~$r$. The $n$-dimensional pseudo-sphere and pseudo-hyperbolic spaces of index $1$ and radius $1$ are known as the \emph{de Sitter} and the \emph{anti de Sitter} spacetimes; they are denoted by $dS^n$ and $AdS^n$, respectively.

The {semi-orthogonal group} is defined as $O(\nu, n- \nu)=\{A\in Gl(n,\bR): A\trsp \epsilon A=\epsilon\}$, where $\epsilon$ is the diagonal matrix $\diag (-I_\nu , I_{n-\nu})$.
The {special semi-orthogonal group} is
$SO(\nu, n-\nu)=\{A\in O(\nu, n-\nu) : \det A=1 \}$.
As pseudo-Riemannian symmetric spaces we can write
\begin{align*}
S^n_{\nu} &{}= O(\nu, n- \nu +1)/O(\nu, n- \nu) = SO(\nu, n- \nu +1)/SO(\nu, n- \nu) , \\
H^n_{\nu} &{}= O(\nu +1, n- \nu)/O(\nu, n- \nu) = SO(\nu +1, n- \nu)/SO(\nu, n- \nu).
\end{align*}

\subsection{Isometric actions}\label{subsec:isometric}\hfill

Let $M$ be a manifold. An action of a Lie group $G$ on $M$ is called proper if the map
\[
G \times M \to M \times M,  \quad (g,p)\mapsto (p,g\cdot p),
\]
is proper, that is, the preimage of a compact subset of $M\times M$ is compact in $G\times M$. If an action is proper, then the orbit space $M/G$ is a Hausdorff space with the quotient topology, the orbits are closed submanifolds of $M$, and all isotropy subgroups are compact.

We now consider a pseudo-Riemannian manifold $(M,g)$. We denote its isometry group by $I(M)$. It is well-known that $I(M)$ is a Lie group. All the actions that we consider in this paper are isometric (that is, they preserve the pseudo-Riemannian structure of $M$). An action is effective if the only element of the group acting as the identity is the identity itself. An effective isometric action is equivalent to the natural action of a subgroup of the isometry group. The action of a connected Lie subgroup $G$ of the isometry group $I(M)$ is called of cohomogeneity $r$ if the minimum codimension of its orbits is $r$. In this case, $M$ is said to be a cohomogeneity $r$ $G$-manifold.

Assume now that $M$ is a cohomogeneity one $G$-manifold, and that the action of $G$ on $M$ is proper. Then, by \cite{Be82}, the orbit space $M/G$ is a one dimensional manifold, possibly with boundary, homeomorphic to (i) $\bR$, (ii) $S^1$, (iii) $[0,1)$, or (iv) $[0,1]$ with their standard topology. We denote by $\pi\colon M\to M/G$ the projection. A point $p \in M$ is called regular (resp.\ {singular}) if $\pi(p)$ is an interior (resp.\ boundary) point. The corresponding orbit $G\cdot p$ is called {principal} or {regular} (resp.\ {singular}). A singular orbit whose codimension coincides with the cohomogeneity is usually called exceptional.
In cases (i) and (iii), the manifold $M$ is $G$-equivariantly diffeomorphic to $G/K \times \bR$ and the twisted product $G\times_{H}V$, respectively, where $K$ and $H$ are the isotropy subgroups of a regular and a singular point of $M$, respectively. The manifold $V$ is $H$-diffeomorphic to a $(d+1)$-dimensional Euclidean vector space upon which $H$ acts linearly and transitively on the unit sphere $S^d \subset V$, and $S^d \cong H/K$. In case (iii) there is exactly one singular orbit, every other orbit is principal, and a tube around the singular orbit. In case (ii) the projection $\pi\colon M \to S^1$ is a fibration with fiber $G/K$. In cases (i) and (ii) all orbits are principal and form a regular cohomogeneity one foliation. Finally, in case (iv), $M$ is obtained by gluing two manifolds of type (iii) together. In this case there are exactly two singular orbits and every other orbit is principal and a tube around any of the singular orbits (cf.~e.g.~\cite{AA92}).

In the Riemannian setting, an effective isometric action of a Lie group $G$ is proper if and only if it corresponds to the standard action of a \emph{closed} subgroup of the isometry group. Thus, cohomogeneity one manifolds are a natural generalization of homogeneous manifolds as cohomogeneity zero $G$-manifolds. In the non-Riemannian setting, a closed subgroup of the isometry group does not necessarily act properly. In this case, there are examples of cohomogeneity zero $G$-manifolds that are not homogeneous, that is, the group $G$ does not act transitively on $M$ (see Proposition~\ref{th:parab-action}, for example).


\section{Cohomogeneity one actions on $\adss$}
\label{sec: C1-LieSub-SU(1,n)}

The aim of this section is to present some examples of cohomogeneity one actions on odd dimensional anti De Sitter spacetimes that are obtained by lifting cohomogeneity one actions of complex hyperbolic spaces via the Hopf fibration.

We first recall some properties of the Lie algebra of the isometry group of $\chn$.

\subsection{Complex hyperbolic spaces}\label{subsec:chn}\hfill

We consider the complex vector space $\bC^n$ endowed with its standard complex structure $J$. Obviously, $\bC^n$ has an underlying structure of a real vector space that is isomorphic to $\bR^{2n}$. Thus, a \textit{real subspace} of $\bC^n$ is an $\bR$-linear subspace of the real vector space obtained from $\bC^n$ by restricting the scalars to the real numbers. Let $\bV$ be a real subspace of $\bC^n$ . The \emph{K\"{a}hler angle} of a nonzero vector $v \in \bV$ with respect to $\bV$ is defined to be the angle between $Jv$ and $\bV$ or, equivalently, the value $\varphi(v) \in [0,\pi/2]$ such that $\langle \pi_\bV(Jv),\pi_\bV(Jv)\rangle=\cos^2(\varphi(v)) \langle v,v\rangle$, where $\pi_\bV$ denotes the orthogonal projection map onto $\bV$. We say that $\bV$ has constant K\"{a}hler angle $\varphi$ if the K\"{a}hler angle of every nonzero vector $v\in \bV$ with respect to $\bV$ is $\varphi$. In particular, $\bV$ is a complex subspace if and only if it has constant K\"{a}hler angle $0$, and it is a totally real subspace if and only if it has constant K\"{a}hler angle $\pi/2$.

We now consider the complex vector space $\bC^{n+1}$ and denote by $\{e_0,e_1,\dots,e_n\}$ its canonical basis. We denote by $\bC^{1,n}$ the complex vector space $\bC^{n+1}$ endowed with the scalar product
\begin{equation} \label{Re-Lorentz-Pr}
\langle z,w \rangle=\re \left( -z_0\bar{w}_0 + \sum\limits_{k=1}^{n} z_k \bar{w}_k \right), \end{equation}
where $z$, $w\in\bC^{n+1}$. This scalar product makes $\bC^{1,n}$ isometric, as a real vector space, to $\bR^{2,2n}$. Moreover, the odd dimensional anti de Sitter spacetime can now be rewritten as
\[
\adss = \{ z \in \bC^{1,n} : \langle z,z \rangle=-1 \}.
\]

Consider the equivalence relation on $\adss$ generated by $z \sim \lambda z$ with $\lambda \in S^1 \subset \bC$. By definition, the complex hyperbolic space is the quotient manifold $\chn = \adss /\sim$ endowed with the Riemannian metric that makes the projection map $\pi \colon \adss \to \chn$ a pseudo-Riemannian submersion. The complex structure of $\bC^{1,n}$ induces a complex structure that makes $\chn$ a K\"{a}hler manifold of constant negative holomorphic sectional curvature. In particular, $\pi \colon \adss \to \chn$ is a principal fiber bundle over $\chn$ with total space $\adss$, fiber $S^1$ and whose projection map is called the {Hopf map} of $\chn$. If $n=1$, then $\bC H^1$ is isometric to the real hyperbolic space $\bR H^2$; hence, we consider $n\geq 2$ throughout this section.

The expression inside the brackets in~\eqref{Re-Lorentz-Pr} defines a pseudo-hermitian product in $\bC^{n+1}$. The group of transformations that preserves it, is denoted by $U(1,n)$. More explicitly, $U(1,n)=\{A\in GL(n,\bC):A^*\epsilon A=\epsilon\}$, where in this case $\epsilon=\diag(-1,I_n)$, and $(\cdot)^*$ denotes conjugate transpose. We also denote $SU(1,n)=\{A\in U(1,n):\det A=1\}$. It turns out that both $U(1,n)$ and $SU(1,n)$ act transitively but not effectively on $\chn$, the former with kernel $S^1=\{\lambda I_n\in U(1,n):\lambda\in S^1\}$, and the latter with finite kernel. Moreover, the corresponding Lie algebras $\g{u}(1,n)$ and $\g{su}(1,n)$ are reductive and simple, respectively. As a symmetric space, the complex hyperbolic space can be written as $\chn=G/K$ with $G=SU(1,n)$ and $K=S(U(1)U(n))\cong U(n)$.

We need a finer description of $\g{g}=\g{su}(1,n)$ for our results. We use the following notation to denote certain type of matrices that appear in this section:
\[
\lceil t, v, X \rceil =
\left(\begin{array}{@{}c|c@{}}
\ii t & 	v^* \\ \hline
v & X		
\end{array}\right),
\]
where $t \in \bR$, $v \in \bC^n$, and $X$ is a complex $(n \times n)$ matrix. Then, the Lie algebra of the group $SU(1,n)$ can be written as
\[
\g{g}=\g{su}(1,n) = \{\lceil t, v, X \rceil : t \in \bR, v \in \bC^n, X \in \g{u}(n), \ \ii t + \tr X = 0 \} .
\]
Let $\g{k}$ be the Lie subalgebra $\g{s}(\g{u}(1) \oplus \g{u}(n))$ of the Lie algebra $\g{g}$, and $\g{p}$ the orthogonal complement of $\g{k}$ in $\g{g}$ with respect to the Killing form $\B$ of $\g{g}$. Then, $\g{g}=\g{k} \oplus \g{p}$ is the Cartan decomposition of $\g{g}$ corresponding to the Cartan involution $X \mapsto -X^*$. We take the maximal abelian subspace $\g{a}=\bR \lceil 0, e_1, 0 \rceil$ of $\g{p}$ and denote its dual vector space by $\g{a}^*$. For each $\lambda \in \g{a}^*$ we define the subspace $\g{g}_\lambda \subset \g{g}$ by
\[
\g{g}_\lambda = \{ X \in \g{g} : \ad(H)(X)=\lambda(H)X \ \text{for all} \ H \in \g{a} \} .
\]
A nonzero covector $\lambda \in \g{a}^*$ is called a restricted root if the corresponding subspace $\g{g}_\lambda$ is nontrivial. The restricted root space decomposition of $\g{g}$ with respect to $\g{a}$ has the form $\g{g}=\g{g}_{-2\alpha} \oplus \g{g}_{-\alpha} \oplus \g{g}_0 \oplus \g{g}_\alpha \oplus \g{g}_{2\alpha}$ for a certain $\alpha \in \g{a}^*$.

The root spaces $\g{g}_\alpha$ and $\g{g}_{2\alpha}$ are isomorphic to the $(n-1)$-dimensional complex vector space $\bC^{n-1}$ and the $1$-dimensional euclidean vector space $\bR$, respectively. Explicitly,
\begin{align*}
\g{g}_\alpha &{}=
\left\{\lceil 0 , (0,v), \left(\begin{array}{@{}c|c@{}} 0 & v^* \\ \hline -v & 0 \end{array} \right) \rceil : v \in \bC^{n-1} \right\}, &
\g{g}_{2\alpha} &{}=
\left\{\lceil \mu , (\ii \mu , 0), \left(\begin{array}{@{}c|c@{}} -\ii \mu & 0 \\ \hline 0 & 0 \end{array}\right) \rceil : \mu \in \bR \right\}.
\end{align*}
We fix a criterion of positivity on the set of roots by letting $\alpha$ be a positive root. Then the subspace $\g{n}=\g{g}_\alpha \oplus \g{g}_{2\alpha}$, as the sum of the root spaces corresponding to all positive roots, turns out to be a nilpotent Lie subalgebra of $\g{g}$ with center $\g{g}_{2\alpha}$. The decomposition $\g{g}=\g{k} \oplus \g{a} \oplus \g{n}$ is called the Iwasawa decomposition of $\g{g}$.
The corresponding Iwasawa decomposition at Lie group level is $SU(1,n)=KAN$, where $K$, $A$ and $N$, denote the connected Lie subgroups of $SU(1,n)$ whose Lie algebras are $\g{k}$, $\g{a}$ and $\g{n}$, respectively.

Finally, $\g{g}_0=\g{k}_0\oplus\g{a}$, where $\g{k}_0=\g{g}_0\cap\g{k}\cong\g{s}(\g{u}(1)\g{u}(n-1))$ is a Lie subalgebra of $\g{k}$. Explicitly,
\[
\g{k}_0=\left\{\lceil \mu , 0, \left(\begin{array}{@{}c|c@{}} \ii \mu & 0 \\ \hline 0 & Y \end{array}\right)\rceil:\mu\in\bR,\ Y\in\g{u}(n-1),\ 2\ii\mu+\tr Y=0\right\}.
\]
Moreover, $[\g{k}_0,\g{g}_\lambda]\subset\g{g}_\lambda$ for any $\lambda\in\g{a}^*$. If $K_0$ is the connected subgroup of $SU(1,n)$ whose Lie algebra is $\g{k}_0$, then $K_0 A$ is a semi-direct product of Lie groups isomorphic to $S(U(1)U(n-1))\bR$.

\subsection{Lie subgroups of $U(1,n)$ acting with cohomogeneity one on $\adss$}\label{subsec:u1n}\hfill

The idea of this subsection is to lift cohomogeneity one actions on complex hyperbolic spaces to odd dimensional anti de Sitter spacetimes via the Hopf fibration $\pi\colon\adss \to \chn$. We classify, up to orbit equivalence, cohomogeneity one actions of connected, closed Lie subgroups of $U(1,n)$ on $\adss$.

\begin{theorem} \label{Th.-C1-Liesub-SU(1,n)}
Let $H\subset U(1,n)$, $n\geq 2$, be a closed, connected Lie subgroup that acts with cohomogeneity one on $\adss$. Then the action of $H$ is orbit equivalent to one of the following:
\begin{enumerate}[{\rm (1)}]
\item The action of a subgroup of the form $FN$, where
    \begin{enumerate}[{\rm (a)}]
    \item $F=A$, or\label{th:FN:A}
    \item $F=K_0$, or\label{th:FN:K0}
    \item $F=F_c$, where the $F_c$ is generated by exponentiating the matrix
    \[
    X_c=\left(\begin{array}{@{}c|c|c@{}}
    {\ii c} & 1 & 0\\
    \hline
    1 & {\ii c} & 0\\
    \hline
    0 & 0 & 0
    \end{array}\right),
    \text{ with $c\in\bR$, $c\neq 0$.}
    \]\label{th:FN:Fc}
    \end{enumerate}
    In this case all orbits are principal and the orbit space $\adss / H$ is $S^1$.\label{th:FN}

\item The action of $S(U(1,k) \times U(n-k))$, $k \in \{ 0, 1, \ldots , n-1\}$; there is one singular orbit which is isometric to $AdS^{2k+1}$ and every principal orbit is a tube around the singular orbit.\label{th:su1k}

\item The action of $S^1 SO^0(1,n)$; there is one singular orbit obtained as $S^1\cdot(\bR^{1,n}\cap\adss)$, where $\bR^{1,n}=\spann_\bR\{e_0,\dots,e_n\}$, and the rest of the orbits are tubes around this one.\label{th:so1n}

\item The action of $N^0_K(S)S$, where $S \subset SU(1,n)$ is the connected Lie subgroup with the Lie algebra $\g{s}=\g{a} \oplus \g{w} \oplus \g{g}_{2\alpha}$, and $\g{w}$ is a real subspace of $\g{g}_\alpha \cong \bC^{n-1}$ such that $\g{w}^\perp$ is totally real; the orbit through $e_0$ is the intersection $\bW \cap \adss$, where
\[
\bW  = \spann_{\bC} \{e_0, e_1, e_2,\dots, e_{n-r}\}\oplus\spann_\bR\{e_{n-r+1}, \dots, e_{n} \},
\]
and the other orbits are tubes around this orbit.\label{th:lohnherr}

\item The action of $N^0_K(S)S$, where $S \subset SU(1,n)$ is the connected Lie subgroup with the Lie algebra $\g{s}=\g{a} \oplus \g{w} \oplus \g{g}_{2\alpha}$, and $\g{w}$ is a linear subspace of $\g{g}_\alpha \cong \bC^{n-1}$ such that $\g{w}^\perp$ has constant K\"{a}hler angle $\varphi\in (0,\frac{\pi}{2})$; the orbit through $e_0$ is singular and corresponds to the intersection $\bW \cap \adss$ where
\[
\bW  = \spann_{\bC} \{e_0, \dots, e_{k+1}\}\oplus\spann_\bR\{f_1, \dots, f_\ell, h_1, \dots, h_\ell \},
\]
and
\begin{align*}
f_j={}&\cos(\tfrac{\varphi}{2})e_{k+j+1}+\sin(\tfrac{\varphi}{2})\ii e_{k+\ell+j+1},\\
h_j={}&\cos(\tfrac{\varphi}{2})\ii e_{k+j+1}+\sin(\tfrac{\varphi}{2})e_{k+\ell+j+1},
&j=1,\dots,\ell,
\end{align*}
being $\{e_0,\dots,e_n\}$ the canonical basis of $\bC^{1,n}$.
The other orbits are tubes around the orbit through $e_0$.\label{th:berndt-bruck}
\end{enumerate}
\end{theorem}

The rest of this subsection is devoted to the proof of Theorem~\ref{Th.-C1-Liesub-SU(1,n)}. We begin with some general remarks.

Let $\pi\colon \adss \to \chn$ denote the Hopf map, and assume that $H$ is a closed, connected Lie subgroup of $U(1,n)$ acting with cohomogeneity one on $\adss$. Note that, since $H$ is  closed in $U(1,n)$ and $S^1$ is compact, $H$ acts properly on $\chn$. Now, let $\{g_n\}$ and $\{p_n\}$ be sequences in $H$ and $\ads$ respectively, such that $p_n \to p$ and $g_n(p_n)\to q$ in $\adss$. Then the sequences $\{\pi(p_n)\}$ and $\{g_n(\pi(p_n))\}=\{\pi(g_n(p_n))\}$ are convergent in $\chn$ and, by the fact that $H$ acts properly on $\chn$, it follows  that $\{g_n\}$ has a convergent subsequence. In particular, we conclude that $H$ acts properly on $\adss$.

Since the elements of $H$ are $\bC$-linear,  for each orbit $H\cdot p$, either $H\cdot p$ contains none of the fibers at its points, or $H\cdot p$ contains the fiber $S^1\cdot q$ at every point $q \in H\cdot p$.
\medskip

First, we assume that there is a principal orbit $H\cdot p$ that contains none of the fibers. Then the restriction map $\pi|_{H\cdot p}\colon H\cdot p \to \chn$ is (at least locally) a one-to-one smooth map. It follows that the orbit $H\cdot \pi(p)=\pi(H\cdot p)$ is of dimension $2n$, which means that $H$ acts with cohomogeneity zero, hence transitively on $\chn$ (because the action is proper). Now we have~\cite[Theorem~4.2]{CaGaSw09}

\begin{theorem}\label{Trans-CHn}
The connected groups acting transitively on $\chn$ are the full isometry group $SU(1,n)$ and the groups $FN$, where $N$ is the nilpotent factor of the Iwasawa decomposition of $SU(1,n)$ and $F$ is a connected closed Lie subgroup of $K_0 A$ with nontrivial projection onto $A$.
\end{theorem}

It follows from Theorem~\ref{Trans-CHn} that all the orbits are principal, diffeomorphic to $\chn$, and that the orbit space $\adss /H$ is homeomorphic to the circle $S^1$.

Let us consider a subgroup $FN$, as in the statement of Theorem~\ref{Trans-CHn}, acting with cohomogeneity one on the anti de Sitter spacetime $\adss$. We denote by $\g{f}$ the Lie algebra of $F$, which is a Lie subalgebra of $\g{k}_0 \oplus \g{a}$. We also denote by $\g{f}_{\g{a}}$ the projection of $\g{f}$ onto the subspace~$\g{a}$.

By Theorem~\ref{Trans-CHn}, $\g{f}_{\g{a}}=\g{a}$. Hence, we can write $\g{f}=\bR \xi_X\oplus(\g{k}_0\cap\g{f})$, where
\[
\xi_{X}=\lceil c,e_1,
\left(\begin{array}{@{}c|c@{}}
\ii c & 0\\
\hline
0 & X
\end{array}\right)\rceil,
\]
with $c\in\bR$, $X\in\g{u}(n-1)$ and $2c \ii +\tr X =0$. More specifically, we can write
\begin{equation}\label{eq:TpFNp}
\g{f}\oplus\g{n}
=\left\{
\!\left(\begin{array}{@{}c|c|c@{}}
\ii ac+\ii x +\ii y& a-\ii x & v^*\\
\hline
a+\ii x & \ii ac-\ii x+\ii y & v^*\\
\hline
v & -v & aX+Y
\end{array}\right):
\begin{aligned}
&a,x\in\bR,v\in\bC^{n-1},\\
&\lceil y,0,
\left(\begin{array}{@{}c|c@{}}
\ii y & 0\\
\hline
0 & Y
\end{array}\right)\rceil\in\g{k}_0\cap\g{f}
\end{aligned}
\right\}.
\end{equation}

It readily follows from this expression that $T_{e_0}(FN\cdot e_0)$ is at least $2n$-dimensional, and $T_{e_0}(FN\cdot e_0)$ is not $(2n+1)$-dimensional if and only if $y=0$. Thus, $FN$ acts with cohomogeneity one if and only if $y=0$ in~\eqref{eq:TpFNp}; equivalently, $Y\in\g{su}(n-1)$ in~\eqref{eq:TpFNp}.

We take the matrix $X_c$ as in the statement of Theorem~\ref{Th.-C1-Liesub-SU(1,n)}~(\ref{th:FN:Fc}) and define $\g{f}_c=\bR X_c$. We denote by $F_c$ the connected subgroup of $K_0A$ whose Lie algebra is $\g{f}_c$. We now prove that $FN$ has the same orbits as $F_cN$. It will actually be enough to show that $(\g{f}\oplus\g{n})\cdot\tilde{p}=(\g{f}_c\oplus\g{n})\cdot\tilde{p}$ for each $\tilde{p}\in\adss$, by uniqueness of integral submanifolds of an integrable distribution.

Let $\tilde{p}=(p_0,p_1,p)\in\adss$.
A generic element of $T_{\tilde{p}}(FN\cdot\tilde{p})=(\g{f}\oplus\g{n})\cdot\tilde{p}$, according to~\eqref{eq:TpFNp}, is written as
\[
\begin{pmatrix}
a(\ii cp_0+p_1)+\ii x(p_0-p_1)+v^*p\\
a(p_0+\ii c p_1)+\ii x(p_0-p_1)+v^*p\\
(p_0-p_1)v+aXp+Yp
\end{pmatrix}.
\]
On the other hand, we can write $\g{f}_c\oplus\g{n}$ as a particular case of ~\eqref{eq:TpFNp}, just by setting $y=0$, and $X=Y=0$. Now we take
\begin{align*}
y&{}=x-\ii\frac{ap^* Xp+p^*Yp}{\lvert p_0-p_1\rvert^2}\in\bR,&
w&{}=v+\frac{a}{p_0-p_1}Xp+\frac{1}{p_0-p_1}Yp\in\bC^{n-1}.
\end{align*}
It is worthwhile to note here that, for any element $X\in\g{u}(n-1)$, since $X^*=-X$ we get that $p^*Xp$ is a purely imaginary number. Taking the previous expressions into account, it follows that the previous vector equals
\[
\begin{pmatrix}
a(\ii cp_0+p_1)+\ii y(p_0-p_1)+w^*p\\
a(p_0+\ii c p_1)+\ii y(p_0-p_1)+w^*p\\
(p_0-p_1)w
\end{pmatrix}.
\]
Altogether this shows that $(\g{f}\oplus\g{n})\cdot\tilde{p}=(\g{f}_c\oplus\g{n})\cdot\tilde{p}$ as desired.
If $c=0$, then $\g{f}_c=\g{a}$ and we have that the action of $FN$ is orbit equivalent to the action of $AN$. This corresponds to Theorem~\ref{Th.-C1-Liesub-SU(1,n)}~(\ref{th:FN:A}). If $c\neq 0$, the action of $FN$ is orbit equivalent to the action of $F_cN$ on $\adss$ as in Theorem~\ref{Th.-C1-Liesub-SU(1,n)}~(\ref{th:FN:Fc}).

\medskip

In what follows we assume that every principal orbit $H\cdot p$ contains the fiber $S^1\cdot q$ for all $q\in H\cdot p$. Since, by the assumption, $H$ acts with cohomogeneity one on $\adss$ it follows that the action of $H$ on $\chn$ is also of cohomogeneity one. So, in order to complete the proof of Theorem~\ref{Th.-C1-Liesub-SU(1,n)} we have to lift cohomogeneity one actions on $\chn$ to cohomogeneity one actions on $\adss$.
The following theorem by Berndt and Tamaru \cite{BeTa07} gives a complete classification of cohomogeneity one actions on $\chn$, up to orbit equivalence.

\begin{theorem}\label{C1-CHn}
A real hypersurface in $\chn$, $n\geq 2$, is homogeneous if and only if it is holomorphically congruent to one of the following hypersurfaces:
\begin{enumerate}[{\rm (i)}]
\item a tube around a totally geodesic $\chk$, for some $k\in \{ 0, \ldots , n-1 \}$,\label{C1-CHn:1}
\item a tube around a totally geodesic $\bR H^n$,\label{C1-CHn:2}
\item a horosphere,\label{C1-CHn:3}
\item the Lohnherr hypersurface $W^{2n-1}$ or one of its equidistant hypersurfaces,\label{C1-CHn:4}
\item a tube around a Berndt-Bruck submanifold $W^{2n-k}_\varphi$ for some $\varphi \in (0, \pi/2]$ and some $k \in \{ 2, \ldots, n-1 \}$, where $k$ is even if $\varphi \neq \pi/2$.\label{C1-CHn:5}
\end{enumerate}
\end{theorem}

The connected closed subgroups of $SU(1, n)$ that give rise to each one of the cohomogeneity one actions (up to orbit equivalence) are: (\ref{C1-CHn:1}) $S(U(1, k) \times U(n-k))$, (\ref{C1-CHn:2}) $SO^0(1, n)$, (\ref{C1-CHn:3}) $N$, (\ref{C1-CHn:4}) the connected Lie subgroup $S$ of $AN$ whose Lie algebra is $\g{s} = \g{a} \oplus \g{w} \oplus \g{g}_{2\alpha}$, where $\g{w}$ is a linear hyperplane of $\g{g}_\alpha$, (\ref{C1-CHn:5}) $N^0_K(S)S$, where $S$ is the connected Lie subgroup of $AN$ whose Lie algebra is $\g{s} = \g{a} \oplus \g{w} \oplus \g{g}_{2\alpha}$, with $\g{w}$ a real subspace of $\g{g}_\alpha$ such that $\g{w}^\perp=\g{g}_\alpha \ominus \g{w}$ (orthogonal complement of $\g{w}$ in $\g{g}_\alpha$) has dimension $k$ and constant K\"{a}hler angle $\varphi$, and $N_K^0(S)$ is the connected component of the identity of the normalizer of $S$ in~$K$.

Examples~(\ref{C1-CHn:1}) and~(\ref{C1-CHn:2}) in Theorem~\ref{C1-CHn}, correspond to cohomogeneity one actions with one totally geodesic singular orbit. The families~(\ref{C1-CHn:3}) and~(\ref{C1-CHn:4}) provide homogeneous regular foliations on $\chn$ and, hence, the corresponding cohomogeneity one actions do not have singular orbits. The families in~(\ref{C1-CHn:5}) correspond to cohomogeneity one actions with one non-totally geodesic singular orbit.


Before starting our study of the cohomogeneity one actions on odd dimensional anti de Sitter spacetimes that are related to cohomogeneity one actions on complex hyperbolic spaces, we present some general remarks that will be useful afterwards.

Note that since the action of $U(1,n)$ on $\chn$ is given by $g \pi(p)=\pi(g p)$, for any Lie subgroup $H \subset U(1,n)$ and any $p \in \adss$ we have $H\cdot p \subset \pi^{-1}\left(H\cdot \pi(p)\right)$. On the other hand, if the principal orbits of $H$ in $\adss$ contain the fiber at every point, then $\pi^{-1}\left(H\cdot \pi(p)\right) \subset H\cdot p$, that is, $H\cdot p = \pi^{-1}\left(H\cdot \pi(p)\right)$ for any regular $p$. Now let $p\in\adss$ be any point and $\lambda\in S^1$. We can find a sequence of regular points $\{p_n\}$ such that $p_n\to p$. Since $\lambda p_n\in S^1\cdot p_n\subset G\cdot p_n$ by assumption, we get that there exists a sequence $\{g_n\}\subset H$ such that $\lambda p_n=g_n p_n$. Thus, $p_n\to p$, $g_n p_n\to \lambda p$. Since $H$ acts properly, $\{g_n\}$ has a convergent subsequence $\{g_{n_k}\}$ converging to a point $g\in H$. Then, $\lambda p =\lim \lambda p_{n_k}=\lim g_{n_k}p_{n_k}=gp\in H\cdot p$. Therefore, we have proved
\begin{equation}\label{eq:fibers}
H\cdot p =\pi^{-1}(H\cdot\pi(p)),\quad\text{for all $p\in\adss$.}
\end{equation}

Suppose that the actions of two Lie subgroups $H$ and $H'$ of $U(1,n)$ on $\chn$ are orbit equivalent via an isometry $f \in U(1,n)$ which is, in particular, a $\bC$-linear map. Then $f \circ \pi = \pi \circ f$, where $f$ is considered both as an isometry on $\adss$ and an isometry on $\chn$. Hence, the actions of the groups $H$ and $H'$ are orbit equivalent on $\adss$. Thus, in order to lift a cohomogeneity one action of a Lie subgroup $H \subset U(1,n)$ on $\chn$ to a cohomogeneity one action on $\adss$, without loss of generality, one can take $H$ itself as the acting group upon $\adss$.

Now we consider the subgroup $S^1=\{\lambda I:\lambda\in S^1\subset\bC\}$ of $U(1,n)$. Since the elements of $S^1$ are diagonal matrices, they commute with any element of $U(1,n)$. In particular, if $H$ is a subgroup of $U(1,n)$, so is $S^1H$. Since $S^1$ is compact, it follows that $S^1 H$ is closed in $U(1,n)$ and the action of $S^1 H$ on $\adss$ is still proper. Furthermore, the orbits of $S^1 H$ on $\adss$ satisfy~\eqref{eq:fibers}, that is, all the orbits contain the fibers. Therefore, for any subgroup $H$ of $U(1,n)$ acting with cohomogeneity one on the complex hyperbolic space $\chn$, the group $S^1 H$ acts properly on the anti de Sitter spacetime $\adss$ with cohomogeneity one, and each orbit contains the $S^1$-fiber.

Let $\bar{M}$ be a Lorentzian manifold (in particular the anti de Sitter spacetime) and let $M\subset \bar{M}$ be a Lorentzian submanifold of $\bar{M}$, that is, a submanifold such that the induced metric of $\bar{M}$ turns out to be Lorentzian. Under these circumstances, the normal bundle $\nu M$ of $M$ has a positive definite metric at each point. For fixed $r> 0$, we define the tube of radius $r$ around~$M$ as the set
\[
M^r= \{ \exp(r\xi) : \xi \in \nu M,\ \langle\xi,\xi\rangle = 1 \}.
\]
In general, the subset $M^r$ does not need to be a regular submanifold of $\bar{M}$. However, under mild conditions it can be proved that it is a hypersurface. In our case, the submanifold $M$ will be a homogeneous submanifold of $\adss$, and tubes around it will indeed be hypersurfaces.

In what follows, we will apply this method to lift the cohomogeneity one actions on the complex hyperbolic space~$\chn$ given by Theorem~\ref{C1-CHn} to the anti de Sitter spacetime~$\adss$.

\subsubsection*{Lifting the action of $S(U(1,k) \times U(n-k))$}

Clearly, the action of $S(U(1,k) \times U(n-k))$, $k \in \{ 0, 1, \ldots , n-1\}$, on the anti de Sitter spacetime $\adss$ is of cohomogeneity one and this action has exactly one singular orbit which is isometric to $AdS^{2k+1}$. Every principal orbit is a tube around this totally geodesic singular orbit. This corresponds to case~(\ref{th:su1k}) of Theorem~\ref{Th.-C1-Liesub-SU(1,n)}.

\subsubsection*{Lifting the action of $SO^o(1,n)$}

We now consider the group $H=S^1 SO^0(1,n)$ acting on $\adss$. We have that $SO^0(1,n)\cdot e_0=\{(x_0,x_1,\dots,x_n\}\in\bR^{1,n}:-x_0^2+x_1^2+\dots +x_n^2=-1\}$ is a totally real, totally geodesic, anti de Sitter spacetime of dimension $n$ in $\adss$, which is obtained by intersecting the subspace $\bV=\bR^{1,n}=\spann_\bR\{e_0,e_1,\dots,e_n\}$ with $\adss$. Hence \begin{align*}
H\cdot e_0&{}=S^1(\bV\cap\adss)\\
&{}=\{\lambda(x_0,x_1,\dots,x_n)\in\bC^{1,n}:\lambda\in S^1\subset\bC,\ x_i\in\bR,\ -x_0^2+x_1^2+\dots+x_n^2=-1\}.
\end{align*}
Since the action is of cohomogeneity one, the rest of the orbits are tubes around this one, and we obtain~(\ref{th:so1n}) of Theorem~\ref{Th.-C1-Liesub-SU(1,n)}.

\subsubsection*{Lifting the action of $N$}

The complex hyperbolic space $\chn$ is diffeomorphic to $AN$ and the space $\adss$ is a principal bundle over $\chn$ with fiber $S^1$. So, in order to determine all orbits of $N$ in $\adss$ it is sufficient to find the orbits of $N$ through the points in the fibers over an orbit of $A$ in $\chn$. Then, we will get the orbits of $S^1 N$ just by letting $S^1$ act on the orbits of $N$.

First we write down the explicit expression of the elements of the groups $A$ and $N$ as the image of matrices in $\g{a}$ and $\g{n}$ under the exponential map:
\begin{equation}\label{eq:ExpAN}
\begin{aligned}
\Exp(\lceil 0,(x,0),0\rceil)&{}=
\left(\begin{array}{@{}cc|c@{}}
\cosh(x)& \sinh(x)	&  0 \\
\sinh(x)& \cosh(x)	&  0 \\ \hline
0 & 0  & I
\end{array}\right),\\
\Exp(\lceil\mu,(\ii\mu,v),\left(\begin{array}{@{}c|c@{}} -\ii\mu & v^* \\ \hline -v & 0 \end{array}\right)\rceil)&{}=
I +
\left(\begin{array}{@{}cc|c@{}}
\ii\mu + \frac{1}{2}\langle v,v \rangle	& -\ii\mu - \frac{1}{2}\langle v,v \rangle &  v^*	\\
\ii\mu + \frac{1}{2}\langle v,v \rangle	& -\ii\mu - \frac{1}{2}\langle v,v \rangle &  v^*	\\ \hline
v & -v & 0
\end{array}\right),
\end{aligned}
\end{equation}
where $x$, $\mu\in\bR$, and $v\in\bC^{n-1}$.

Using~\eqref{eq:ExpAN} we get
$A\cdot e_0= \{ \cosh(x)e_0 + \sinh(x)e_1 : x \in \bR \}$.
Thus, any point in the $S^1$-fiber at a point of the orbit $A\cdot \pi(e_0)$ is of the form
\[
p_{x,\lambda}=\lambda (\cosh(x)e_0 + \sinh(x)e_1),
\]
for some $x \in \bR$ and $\lambda \in S^1$. We find the orbits of $N$ through the points $p_{x,\lambda}$. Let $\bV \subset \bC^{1,n}$ be the linear subspace
\[
\bV=\spann_\bR \{e_0+e_1, \ii (e_0+e_1), e_2, \ii e_2, \ldots , e_{n}, \ii e_{n}\}.
\]
Using~\eqref{eq:ExpAN} once more,
\begin{align*}
N\cdot p_{x,\lambda}	&{}=p_{x,\lambda}+ \lambda e^{-x} \Bigl\{ \mu \ii (e_0 + e_1) + \frac{1}{2}\langle v,v \rangle (e_0 + e_1) + v : \mu \in \bR,\ v \in \bC^{n-1}\Bigr\} \in p_{x,\lambda}+\bV.
\end{align*}
So, $N\cdot p_{x,\lambda} \subset (p_{x,\lambda}+\bV) \cap \adss$. For the reverse inclusion, let $q = p_{x,\lambda} + z(e_0+e_1) + w \in \adss \cap (p_{x,\lambda}+\bV)$,
with $z\in \bC$, $w\in\spann_\bC\{e_2,\dots,e_n\}\cong\bC^{n-1}$. Taking
\begin{align*}
\mu &{}= e^x\mathop{\rm Im} (z\bar{\lambda}),&
v&{}=e^x\bar{\lambda}w,&
X&{}=\lceil\mu,(\ii\mu,v),\left(\begin{array}{@{}c|c@{}} -\ii\mu & v^* \\ \hline -v & 0 \end{array}\right)\rceil,
\end{align*}
we get $q=\Exp(X)\cdot p_{x,\lambda}$. Thus
$N\cdot p_{x,\lambda} = \adss \cap (p_{x,\lambda} + \bV)$, for all $x \in \bR$, $\lambda \in S^1$, and therefore
\begin{equation} \label{N-Orbit-SU(1,n)}
N\cdot p = \adss \cap (p + \bV), \quad \text{for all $p\in \adss$.}
\end{equation}

The above description of the orbits of $N$, in particular, implies that $N$ acts with cohomogeneity $2$ on $\adss$. Recall that the subgroup $K_0$ whose Lie algebra is $\g{k}_0$ is isomorphic to $S(U(1)U(n-1))$. More explicitly, we can write
\[
K_0=	
\left\{
\begin{pmatrix}
\ii\mu	& & \\
 & \ii\mu & \\
 & & X
\end{pmatrix}
: \mu \in S^1,\ X \in U(n-1),\ \mu^2\det X =1
\right\}.
\]
It is then easy to check that $S^1N\cdot p=K_0 N\cdot p$ for all $p\in\adss$. Hence, the action of $S^1N$ on $\adss$, which is obtained by lifting the action of $N$ on $\chn$ to the anti de Sitter spacetime, corresponds to Theorem~\ref{Th.-C1-Liesub-SU(1,n)}~(\ref{th:FN:K0}).

\subsubsection*{Lifting the tubes around the Berndt-Br\"{u}ck submanifolds of totally real normal bundle}

Let $S$ be the connected Lie subgroup of $SU(1,n)$ whose Lie algebra is $\g{s}=\g{a} \oplus \g{w} \oplus \g{g}_{2\alpha}$, where $\g{w}$ is a real subspace of $\g{g}_\alpha \cong \bC^{n-1}$ such that $\g{w}^\perp$, the orthogonal complement of $\g{w}$ in $\g{g}_\alpha$, is a totally real subspace.
The group $N^0_K(S)S$, where $N^0_K(S)$ is the identity component of the normalizer of $S$ in $K$, acts with cohomogeneity one on $\chn$. The orbits of $S$ form a polar foliation in $\chn$~\cite{DiDoKo13}. In particular, if $\g{w}^\perp$ is $1$-dimensional, equivalently if $\g{w}$ is a hyperplane, the action of $S$ produces the ruled Lohnherr hypersurface and its equidistant hypersurfaces, which is a particular case of this construction that gives rise to the cohomogeneity one action described in Theorem~\ref{C1-CHn}(\ref{C1-CHn:4}). The calculations carried out here work for the actions described in Theorem~\ref{C1-CHn}(\ref{C1-CHn:4}) and Theorem~\ref{C1-CHn}(\ref{C1-CHn:5}) whenever $\varphi=\pi/2$, so we do these two cases simultaneously. In order to determine the orbits of $N^0_K(S)S$ in $\adss$, we first find the orbits of $S$ in $\adss$ and then consider the action of $N^0_K(S)$ on the orbits of $S$.

The subspace $\g{w}$ decomposes as the orthogonal direct sum of a complex subspace of $\g{g}_\alpha$, and a totally real subspace of $\g{g}_\alpha$~\cite{DiDoKo13}. We may therefore assume, without loss of generality, that $\g{w}\cong\bC^{n-1-r}\oplus\bR^r$. For the purpose of the calculations that follow, we choose a complex basis $\{e_2,\dots,e_n\}$ in such a way that $\g{w}=\g{w}_0\oplus\g{w}_{\pi/2}$, where $\g{w}_0=\spann_\bC\{e_2, \dots, e_{n-r}\}\cong\bC^{n-r-1}$ and $\g{w}_{\pi/2}=\spann_\bR\{e_{n-r+1}, \dots, e_{n} \}\cong\bR^r$, for some $1\leq r\leq n-1$.

An element $X\in\g{s}$ is of the form
\[
X=
\left(
\begin{array}{@{}cc|cc@{}}
\ii\mu 	& x-\ii\mu 	& z^*	& u\trsp 	\\
x+\ii\mu & -\ii\mu 	& z^*	& u\trsp 	\\ \hline
z & -z 	& \multicolumn{2}{c}{\multirow{2}{*}{0}} \\
u& -u	& &	
\end{array}
\right),
\]
where $x$, $\mu \in \bR$, $z \in \bC^{n-r-1}$, and $u \in \bR^r$. Then,
\begin{equation} \label{eX-S-SU(1,n)-I}
\begin{aligned}
\!\Exp(X)={}		
& I + \frac{\cosh(x)-1}{x}\!
\left(
\begin{array}{@{}cc|cc@{}}
x+\frac{1}{x}(\langle z,{z}\rangle+\langle u,u \rangle)	& \frac{1}{x}(\langle z,{z}\rangle+\langle u,u \rangle)			& z^*	& u\trsp	\\
\frac{1}{x}(\langle z,{z}\rangle+\langle u,u \rangle)		& x-\frac{1}{x}(\langle z,{z}\rangle+\langle u,u \rangle)		& z^*	& u\trsp	 \\ \hline
-z& z& \multicolumn{2}{c}{\multirow{2}{*}{0}}\\
-u& u& &
\end{array}
\right)\! \\
& + \frac{\sinh(x)}{x}
\left(
\begin{array}{@{}cc|cc@{}}
\ii\mu 	& x-\ii\mu 	& z^*	& u\trsp	\\
x+\ii\mu & -\ii\mu 	& z^*	& u\trsp	\\ \hline
z& -z	& \multicolumn{2}{c}{\multirow{2}{*}{0}}\\
u& -u	& &
\end{array}
\right)\in S.
\end{aligned}
\end{equation}
Hence, we have
\begin{equation} \label{S.e1-SU(1,n)}
S\cdot e_0=
\left\{	
\begin{pmatrix}
\frac{\cosh(x)-1}{x^2}(\langle z,{z}\rangle+\langle u,u \rangle) + \cosh(x) + \ii\mu \frac{\sinh(x)}{x} \\
\frac{\cosh(x)-1}{x^2}(\langle z,{z}\rangle+\langle u,u \rangle) + \sinh(x) + \ii\mu \frac{\sinh(x)}{x} \\
\frac{1-e^{-x}}{x}z \\
\frac{1-e^{-x}}{x}u
\end{pmatrix}:
\begin{array}{l}
z \in \bC^{n-r-1}, \\
u \in \bR^r, \\
x, \mu \in \bR
\end{array}
\right\}.
\end{equation}
We define the subspace $\bV=\spann_\bR \{ e_0, e_1, \ii (e_0+e_1) \} \oplus \g{w}$. We also define the half-subspaces $\bV_+=\{v \in \bV : \langle v,e_0\rangle <0 \}$ and $\bV_-=\{v \in \bV : \langle v,e_0\rangle >0 \}$ (note that $e_0$ is timelike). Obviously,~\eqref{S.e1-SU(1,n)} implies $S\cdot e_0 \subset \adss \cap \bV_+$. We now check the reverse inclusion.

Let
$q=a_0e_0+a_1e_1+b \ii (e_0+e_1)+w+v \in \adss \cap \bV_+$,
where $a_0$, $a_1$, $b\in\bR$, $w\in\g{w}_{0}$ and $v\in\g{w}_{\pi/2}$. Note that $\langle q,q\rangle=-1$ implies $a_0\geq 1$ and $a_0>a_1$. Then,
\[
\Exp\left(
\begin{array}{@{}cc|cc@{}}
\ii\mu 	& x-\ii\mu 	& z^*	& u\trsp 	\\
x+\ii\mu & -\ii\mu 	& z^*	& u\trsp 	\\ \hline
z & -z 	& \multicolumn{2}{c}{\multirow{2}{*}{0}} \\
u& -u	& &	
\end{array}
\right)\cdot e_0=q,\
\text{ where }
\left\{
\begin{aligned}
x&=-\ln(a_0-a_1), &
\mu &= \frac{x}{\sinh(x)}b, \\
z&=\frac{x}{1-e^{-x}}w, &
u&=\frac{x}{1-e^{-x}}v.
\end{aligned}
\right.
\]
Therefore $S\cdot e_0 = \adss \cap \bV_+$.

In order to determine all the orbits of $S$ on $\adss$, we only need to find the orbits through the points of the form $p=\lambda(x_0e_0 + \ii v)$, for $\lambda \in S^1$ and $v\in\g{w}_{\pi/2}$. We assume $x_0 >0$. Then, taking $g\in S$ as in~\eqref{eX-S-SU(1,n)-I}, and using~\eqref{S.e1-SU(1,n)} we get
\begin{align*}
g\cdot p&{}=g(\lambda(x_0 e_0+\ii v))=\lambda(x_0 g e_0+\ii gv)
=\lambda\Bigl(x_0 g e_0+\frac{e^x-1}{x}\langle u,v\rangle\ii(e_0+e_1)+\ii v\Bigr)\\
&{}=\lambda(x_0 e_0+\ii v)+\lambda\Bigl(x_0 g e_0+\frac{e^x-1}{x}\langle u,v\rangle\ii(e_0+e_1)-x_0e_0\Bigr)\in p+\lambda\bV_+.
\end{align*}
Thus, $S\cdot p\subset(p+\lambda\bV_+)\cap\adss$.

In order to verify that the reverse inclusion holds, since the elements of $S$ are $\bC$-linear, it suffices to take $q=(x_0e_0+\ii v)+(y_0e_0+y_1e_1+b\ii(e_0+e_i)+w+v')\in(p+\bV_+)\cap\adss$ and check that there exists $g\in S$ such that $g(x_0e_0+\ii v)=q$. In the notation of~\eqref{eX-S-SU(1,n)-I} this is achieved by taking
\[
\begin{aligned}
x&=-\frac{\ln(x_0+y_0-y_1)}{x_0}, &
\mu &= \frac{x}{x_0\sinh(x)}\Bigl(b-\frac{e^x-1}{x}\langle u,v\rangle\Bigr), \\
z&=\frac{x}{x_0(1-e^{-x})}w, &
u&=\frac{x}{x_0(1-e^{-x})}v'.
\end{aligned}
\]
Therefore, we have $S\cdot p = p + \lambda\bV_+$ for $p=\lambda(x_0e_0 + v)$ with $x_1 >0$, $\lambda \in S^1$ and $v\in\ii\g{w}_{\pi/2}$. In a similar way, we get $S\cdot p = p + \lambda\bV_-$ if $x_0 < 0$.

Now that we have found the orbits of $S$, in order to describe the orbit of the group $N^0_K(S)S$ through a point $p \in \adss$ we let the group $N^0_K(S)$ act on the orbit $S\cdot p$. Again, we may assume $p=\lambda(x_0e_0 + \ii v)$ with $\lambda\in S^1$, $x_0>0$ and $v\in\g{w}_{\pi/2}$. The identity component of the normalizer of $S$ in $K$ is $N^0_K(S)=S(U(1)U(n-r-1) O(r))$ (see~\cite{DiDoKo13}). We have  $N_K(S)\cdot(x_0 e_0+\ii v)=S^1\cdot e_0+\ii SO(r)\cdot v$, so $\ii (SO(r)\cdot v)$ is an $(r-1)$-dimensional sphere in $\ii\g{w}_{\pi/2}$.  Moreover, $N_K(S)\cdot\bV_\pm=\bW$, where
$\bW=\spann_{\bC}\{e_0,e_1\}\oplus\g{w}$. This implies
\begin{align*}
N^0_K(S)S\cdot p	={}	& N^0_K(S)\cdot (S\cdot p)= N^0_K(S)\cdot \bigl((p + \lambda \bV_+) \cap \adss \bigr)\\
	 ={}	& \bigl(N^0_K(S)\cdot p + \lambda N^0_K(S)\cdot  \bV_+\bigr) \cap \adss \\
	 ={}	& \left(SO(r)\cdot\ii v) +  \bW\right) \cap \adss,
\end{align*}
and hence, the group $N^0_K(S)S$ acts with cohomogeneity one on $\adss$. The orbit through $e_0$ is the intersection $\bW \cap \adss$, and the other orbits are tubes around this one. This is part~(\ref{th:lohnherr}) of Theorem~\ref{Th.-C1-Liesub-SU(1,n)}.

\subsubsection*{Lifting the tubes around the Berndt-Br\"{u}ck submanifolds whose normal bundle has constant K\"{a}hler angle $\varphi\in(0,\pi/2)$}

Let $\g{w}$ be a real linear subspace of $\g{g}_\alpha \cong \bC^{n-1}$ such that $\g{w}^\perp$ has constant K\"{a}hler angle $\varphi\in (0,{\pi}/{2})$. Denote by $S$ the connected Lie subgroup of $AN$ with Lie algebra $\g{s}=\g{a} \oplus \g{w} \oplus \g{g}_{2\alpha}$, and by $N^0_K(S)$ the identity component of its normalizer in $K$. Then the group $N^0_K(S)S$ acts on $\chn$ with cohomogeneity one. For the point $o=\pi(e_0) \in \chn$, we denote by  $W^{2n-r}_\varphi=N^0_K(S)S\cdot o$ the orbit of $N_K^0(S)S$ through $o$. Then, $W^{2n-r}_\varphi$ is of dimension $(2n-r)$, where $r=\dim \g{w}^\perp=2\ell$ is an even number. The other orbits are tubes around this one. More details can be found in~\cite{BeBr01}. However, a more detailed discussion on real subspaces of complex vector spaces can be found in~\cite[Section~2.3]{DiDoKo13}.

We will show that $N_K^0(S)S$ acts on $\adss$ with cohomogeneity one. The calculations that follow are very similar to case~(\ref{th:lohnherr}) of Theorem~\ref{Th.-C1-Liesub-SU(1,n)}, although somewhat more complicated. We point out the differences with the previous case here and skip the routine calculations.

As before, we identify $\g{g}_\alpha$ with a subspace of $\bC^{n-1}$. The subspace $\g{w}$ admits a decomposition $\g{w}=\g{w}_0\oplus\g{w}_\varphi$, where $\g{w}_0$ is a complex subspace of $\g{g}_\alpha$, and $\g{w}_\varphi$ is a real subspace of $\g{g}_\alpha$ with constant K\"{a}hler angle $\varphi\in(0,\pi/2)$. Recall that $\g{w}_\varphi$ has dimension $r=2\ell$, and $\bC\g{w}_\varphi=\g{w}_\varphi\oplus\g{w}^\perp$. We may assume $\g{w}_0=\spann_\bC\{e_2,\dots,e_{k+1}\}$, $\g{w}_\varphi=\spann_\bR \{ f_1, \dots, f_\ell, h_1, \dots, h_\ell \}$,
where $k+2\ell=n-1$ and
\begin{align*}
f_j={}&\cos(\tfrac{\varphi}{2})e_{k+j+1}+\sin(\tfrac{\varphi}{2})\ii e_{k+\ell+j+1},
&h_j={}&\cos(\tfrac{\varphi}{2})\ii e_{k+j+1}+\sin(\tfrac{\varphi}{2})e_{k+\ell+j+1},
&j=1,\dots,\ell.
\end{align*}
Taking into account the previous decomposition of $\g{w}$, and the expression of the vectors $\{f_1,\dots,f_\ell,h_1,\dots,h_\ell\}$ in the basis $\{e_2,\dots,e_n\}$ of $\g{g}_\alpha$, an element of $\g{s}$ can be written as
\[
\left(
\begin{array}{@{}cc|ccc@{}}
\ii\mu & x-\ii\mu & z^*	&  (u-\ii v)^*\cos(\tfrac{\varphi}{2})	& (v-\ii u)^*\sin(\tfrac{\varphi}{2})\\
x+\ii\mu & -\ii\mu & z^* & (u-\ii v)^*\cos(\tfrac{\varphi}{2}) & (v-\ii u)^*\sin(\tfrac{\varphi}{2})\\
\hline
z	& -z & 	&	&	\\
(u+\ii v)\cos(\tfrac{\varphi}{2})	& -(u+\ii v)\cos(\tfrac{\varphi}{2})	 & 	 \multicolumn{3}{c}{0}	 \\
(v+\ii u)\sin(\tfrac{\varphi}{2})	& -(v+\ii u)\sin(\tfrac{\varphi}{2})	 & 	 &	 &
\end{array}\right),
\]
where $x$, $\mu\in\bR$, $z\in\bC^k$, and $u$, $v\in\bR^\ell$.

Proceeding as in the previous part, one can see, after some calculations, that $S\cdot e_0=\bV_+ \cap \adss$ where $\bV = \spann_\bR\{e_0, e_1, \ii (e_0 + e_1) \} \oplus \g{w} \subset \bC^{1,n}$ and $\bV_+=\{v\in\bV:\langle v,e_0\rangle<0\}$ (that is, $\bV_+$ consists of the vectors in $\bV$ whose $e_0$-coefficient is positive).

In order to determine the other orbits of $S$, it is sufficient to find the orbits through the points of the form $p=\lambda (x_0e_0 + v)$, where $\lambda \in S^1$ and $v\in\g{w}^\perp$. For calculation purposes it is convenient to take the basis $\{f_1',\dots,f_\ell',h_1',\dots,h_\ell'\}$ of $\g{w}^\perp$ given by
\begin{align*}
f_j'={}&-\sin(\tfrac{\varphi}{2})e_{k+j+1}+\cos(\tfrac{\varphi}{2})\ii e_{k+\ell+j+1},
&h_j'={}&-\sin(\tfrac{\varphi}{2})\ii e_{k+j+1}+\cos(\tfrac{\varphi}{2})e_{k+\ell+j+1},
\end{align*}
for $j=1,\dots,\ell$. (Recall that $\bC\g{w}_\varphi=\g{w}_\varphi\oplus\g{w}^\perp$, and that $\dim\g{w}_\varphi=\dim\g{w}^\perp$.) For such a point $p$ we can show that $S\cdot p= (p + \lambda\bV_+) \cap \adss$ if $x_0>0$, or $S\cdot p=(p + \lambda\bV_-) \cap \adss$ if $x_0<0$, where $\bV_-$ is defined analogously.

By~\cite[Lemma~2.8]{DiDoKo13}, $N^0_K(S) = S(U(1)U(\g{w}_0)U(\g{w}_\varphi)) \cong U(k) \times U(l)$.
So, for a point $p$ as above, one has
\begin{align*}
N^0_K(S)S\cdot p	={}	& N^0_K(S)\cdot (S\cdot p)= N^0_K(S)\cdot \bigl((p + \lambda \bV_+) \cap \adss \bigr)\\
	 ={}	& \bigl(N^0_K(S)\cdot p + \lambda N^0_K(S)\cdot  \bV_+\bigr) \cap \adss \\
	 ={}	& \left(N^0_K(S)\cdot v) + \bW\right) \cap \adss,
\end{align*}
where $\bW = N^0_K(S)\cdot  \bV_+ = \spann_{\bC} \{e_0, \ii e_1\}\oplus\g{w}$. The vector $v\in\g{w}^\perp$ is also in the orthogonal complement of $\bW$ and
$N^0_K(S)\cdot v\cong U(l)\cdot v$
is a $(2\ell-1)$-dimensional sphere (see~\cite{BeBr01}). Therefore, the orbit of $N^0_K(S)S$ through $e_0$ is the intersection $\bW \cap \adss$, and the other orbits are tubes around this one. This completes the proof of Theorem~\ref{Th.-C1-Liesub-SU(1,n)}.\qed


\section{Examples of cohomogeneity one and cohomogeneity zero actions on anti de Sitter spacetimes} \label{sec:parabolicsubgroups}

In this section we present new examples of cohomogeneity zero and cohomogeneity one actions on anti de Sitter spacetimes. It turns out that there are examples of cohomogeneity zero actions which are not transitive, a phenomenon that was already pointed out in~\cite{BeDiVa}. Many examples of isometric actions are obtained from the restricted root space decomposition of the Lie algebra of $SO^0(2,n)$ which, in turn, is the way to obtain the Iwasawa decompostion of $\g{so}(2,n)$, and the Langlands decompositions of the parabolic subalgebras of $\g{so}(2,n)$.


\subsection{The Lie algebra $\g{so}(2,n)$} \label{Iwa-SO}\hfill

The Lie algebra $\g{g}=\g{so}(2,n)$ of the Lie group $G=SO^0(2 , n)$  is
\begin{align*}
\g{so}(2,n)	& {}= \{ X \in \g{gl}(n+2,\bR) : \epsilon X + X\trsp \epsilon =0, \ \tr X=0 \} \\
		& {}= \Bigl\{ \begin{pmatrix} A & u\trsp \\ u & B \end{pmatrix} : A \in \mathfrak{so}(2), B \in \mathfrak{so}(n), u \in \mathcal{M}_{n\times 2} \Bigr\},
\end{align*}
where $\epsilon = \diag (-I_2,I_n)$, and $\mathcal{M}_{n\times 2}$ denotes the vector space of matrices with $n$ rows and $2$ columns. The real Lie algebra $\g{so}(2,n)$ is simple for $n\geq 3$. Thus, we assume $n\geq 3$ from now on. The Killing form $\B$ on $\g{g}$ is given by
\begin{equation}\label{killingform}
\B (X,Y)=\tr (\ad X \circ \ad Y)=n\tr XY, \quad \text{ for all $X$, $Y \in \g{so}(2,n)$.}
\end{equation}
The Cartan involution $\theta(X)=-X\trsp$ gives rise to the Cartan decomposition $\g{so}(2,n) = \g{k} \oplus \g{p}$, where
\begin{align*}
\g{k}&{}=\left\{
\begin{pmatrix}
A & 0\\
0 & B
\end{pmatrix}:A\in\g{so}(2),B\in\g{so}(n)\right\}\cong\g{so}(2)\oplus\g{so}(n),&
\g{p}&{}=\left\{
\begin{pmatrix}
0 & u\trsp\\
u & 0
\end{pmatrix}:u\in\mathcal{M}_{n\times 2}\right\}.
\end{align*}
The Killing form of $\g{so}(2,n)$ is positive definite on $\g{p}$ and negative definite on $\g{k}$. Thus, the formula $\B_\theta (X,Y)=-\B(X,\theta Y)$ defines an inner product on $\g{g}$.

From now on we make the following choice of a maximal flat $\g{a}$ of $\g{p}$ by setting
\[
\g{a}=\left\{
H_{a,b}=\left(\begin{array}{@{}c|c|c}
0&
\begin{array}{cc}
a&0\\
0&b
\end{array}&0\\
\hline
\begin{array}{cc}
a&0\\
0&b
\end{array}&
0&0\\
\hline
0&0&0
\end{array}\right)\in\g{so}(2,n)
:a,b\in\bR\right\}.
\]
Recall that a restricted root space is defined as $\g{g}_\lambda=\{X\in\g{g}:[H,X]=\lambda(H)X,\forall H\in\g{a}\}$ for each covector $\lambda\in\g{a}^*$, whenever $\g{g}_\lambda$ is nonzero.
We consider the covectors $\alpha_1$, $\alpha_2\in \g{a}^*$ by setting $\alpha_1\left(H_{a,b}\right) =-a+b$, $\alpha_2\left(H_{a,b}\right) =a$. It turns out that $\g{so}(2,n)$ is a real Lie algebra of type $B_2$ whose set of restricted roots is $\Sigma=\{\pm\alpha_1,\pm\alpha_2,\pm(\alpha_1+\alpha_2),
\pm(\alpha_1+2\alpha_2)\}$. We choose a criterion of positivity so that the set of positive roots is $\Sigma^+=\{\alpha_1,\alpha_2,\alpha_1+\alpha_2,
\alpha_1+2\alpha_2\}$. Then, the set of simple roots is precisely $\Lambda=\{\alpha_1,\alpha_2\}$. The root spaces are calculated explicitly in Figure~\ref{fig:root-spaces}.

\begin{figure}[t]
{\scriptsize
\begin{gather*}
\g{g}_{\alpha_1}=
\bR\left(\begin{array}{@{}c|c|c}
\begin{array}{cc}
0&1\\
-1&0
\end{array}&
\begin{array}{cc}
0&-1\\
-1&0
\end{array}&0\\
\hline
\begin{array}{cc}
0&-1\\
-1&0
\end{array}&
\begin{array}{cc}
0&1\\
-1&0
\end{array}&0\\
\hline
0&0&0
\end{array}\right),\qquad
\g{g}_{\alpha_2}=
\left\{\left(\begin{array}{@{}c|c|c}
0&0&
\begin{array}{c}
v\trsp\\ 0
\end{array}\\
\hline
0&0&
\begin{array}{c}
v\trsp\\ 0
\end{array}\\
\hline
\begin{array}{cc}
v& 0
\end{array}&
\begin{array}{cc}
-v&0
\end{array}&0
\end{array}\right):v\in\bR^{n-2}\right\},\\
\g{g}_{\alpha_1+\alpha_2}=
\left\{\left(\begin{array}{@{}c|c|c}
0&0&
\begin{array}{c}
0\\ w\trsp
\end{array}\\
\hline
0&0&
\begin{array}{c}
0\\ w\trsp
\end{array}\\
\hline
\begin{array}{cc}
0& w
\end{array}&
\begin{array}{cc}
0&-w
\end{array}&0
\end{array}\right):v\in\bR^{n-2}\right\},\qquad
\g{g}_{\alpha_1+2\alpha_2}=
\bR\left(\begin{array}{@{}c|c|c}
\begin{array}{cc}
0&1\\
-1&0
\end{array}&
\begin{array}{cc}
0&-1\\
1&0
\end{array}&0\\
\hline
\begin{array}{cc}
0&1\\
-1&0
\end{array}&
\begin{array}{cc}
0&-1\\
1&0
\end{array}&0\\
\hline
0&0&0
\end{array}\right).
\end{gather*}}
\caption{Positive root spaces of $\g{so}(2,n)$ with respect to $\g{a}$.}\label{fig:root-spaces}
\end{figure}

The long simple root of this root system is $\alpha_1$, which has multiplicity one, and the short simple root, $\alpha_2$, has multiplicity $n-2$. Recall that the root spaces corresponding to the negative roots are determined by the equation $\g{g}_{-\lambda}=\theta\g{g}_\lambda$. It is also well known that $[\g{g}_\lambda,\g{g}_\mu]=\g{g}_{\lambda+\mu}$ for any $\lambda$, $\mu\in\g{a}^*$. Finally, $\g{g}_0=\g{k}_0\oplus\g{a}$, where $\g{k}_0=\g{g}_0\cap\g{k}$ is the normalizer of $\g{a}$ in $\g{k}$. In this case $\g{k}_0\cong\g{so}(n-2)$.

We now define $\g{n}$ to be the vector space direct sum of the positive root spaces, that is, $\g{n}=\g{g}_{\alpha_1}\oplus\g{g}_{\alpha_2}
\oplus\g{g}_{\alpha_1+\alpha_2}\oplus\g{g}_{\alpha_1+2\alpha_2}$. Then, $\g{n}$ is a 3-step nilpotent subalgebra of $\g{g}$ and we have the so-called Iwasawa decomposition $\g{g}=\g{k}\oplus\g{a}\oplus\g{n}$. If we denote by $K$, $A$ and $N$ the corresponding connected subgroups of $G$ whose Lie algebras are $\g{k}$, $\g{a}$ and $\g{n}$, respectively, then we also get the Iwasawa decomposition $G=KAN$.


\subsection{Cohomogeneity one actions of Lie subgroups of $AN$ containing $N$}\label{sec:N}\hfill

First, we determine the orbits of the nilpotent factor $N$ in the Iwasawa decomposition $SO^0(2,n)=KAN$. For that matter, let $p=(p_1, \dots, p_{n+2})\in \ads$, that is, $-p_1^2-p_2^2+p_3^2+\dots+p_{n+2}^2=-1$.

The tangent space of the orbit $N\cdot p$ at $p$ is $T_p(N\cdot p)=\g{n}\cdot p$, which can be calculated using the expressions obtained in Subsection~\ref{Iwa-SO}. If $p_2=p_4$, then $\lvert p_1\rvert\neq\lvert p_3\rvert$, and in this case we get $(\g{g}_{\alpha_1}\oplus\g{g}_{\alpha_1+\alpha_2}
\oplus\g{g}_{\alpha_1+2\alpha_2})\cdot p\subset\spann\{e_2+e_4\}$; in particular,
\begin{align}\label{N-sing-Tan-Sp}
\g{n}\cdot p&{}=\spann\{e_2+e_4\}\oplus\g{g}_{\alpha_2}\cdot p,&  \dim(\g{n}\cdot p)&{}=n-1,&
\text{if $p_2=p_4$.}
\end{align}
If $p_2 \neq p_4$, then $\g{g}_{\alpha_2}\cdot p\subset(\g{g}_{\alpha_1}\oplus\g{g}_{\alpha_1+\alpha_2})\cdot p$, and thus,
\begin{align}\label{N-prin-Tan-Sp}
\g{n}\cdot p&{}=(\g{g}_{\alpha_1}\oplus\g{g}_{\alpha_1+\alpha_2}
\oplus\g{g}_{\alpha_1+2\alpha_2})\cdot p,&
\dim(\g{n}\cdot p)&{}=n,&
\text{if $p_2\neq p_4$.}
\end{align}
Therefore, we have proved that $N$ acts with cohomogeneity one on $\ads$.

We now obtain a more explicit form of the orbits of $N$ on $\ads$. According to the results of Figure~\ref{fig:root-spaces}, an element $X\in \g{n}$ can be written as
\begin{equation}\label{gen-X-in-n}
X=\left(
\begin{array}{@{}cc|cc|c}
0 & b+a & 0 & -b-a & {v\trsp} \\
-b-a & 0 & -b+a & 0 & {w\trsp}\\
\hline
0 & -b+a & 0 & b-a & {v\trsp}\\
-b-a & 0 & -b+a & 0 & {w\trsp}\\
\hline
v & w & -v & -w & 0
\end{array}\right),
\text{ where $a$, $b\in\bR$, $v$, $w\in\bR^{n-2}$.}
\end{equation}

If $p_2=p_4$, according to~\eqref{N-sing-Tan-Sp} one may assume $b=0$, $w=0$, and thus
\[
\exp(X)=
\left(\begin{array}{@{}cc|cc|c}
1+\frac{1}{2}\langle v,v\rangle &  a	& -\frac{1}{2}\langle v,v\rangle & -a	 & v\trsp \\
-a & 1	& a & 0 & 0\\
\hline
\frac{1}{2}\langle v,v\rangle	& a & 1-\frac{1}{2}\langle v,v\rangle	& -a	 & v\trsp \\
-a & 0 & a & 1 & 0 \\
\hline
v & 0 & -v & 0	& I_{n-2}
\end{array}\right).
\]
Hence, $(\Exp(X)\cdot p)_2=(\Exp(X)\cdot p)_4$ and $(\Exp(X)\cdot p)_1-(\Exp(X)\cdot p)_3=p_1-p_3$, where subscripts just refer to the corresponding component of the vector. Conversely, let $q\in\ads$ such that $q_2=q_4$ and $q_1-q_3=p_1-p_3$. Then, if we take $X\in\g{n}$ as in~\eqref{gen-X-in-n} such that $a=(p_2-q_2)/(p_1-p_3)\in\bR$ and $v=(q-p)/(p_1-p_3)\in\bR^{n-2}$, then $\Exp(X)\cdot p=q$. Indeed,
\begin{equation}\label{N-orbit}
N\cdot p= \ads \cap (\bV+p),
\text{ where
$\bV=\{ x \in \bR^{2,n} : x_1=x_3, x_2=x_4 \}$, if $p_2=p_4$.}
\end{equation}

Now, assume $p_2 \neq p_4$. By~\eqref{N-prin-Tan-Sp} we may assume $v=0$ in~\eqref{gen-X-in-n}, and thus
\[
\exp(X)=
\left(\begin{array}{@{}cc|cc|c}
1 &  a+b	& o & -a-b	& 0 \\
-a-b & 1-2ab+\frac{1}{2}\langle w,w\rangle	& a-b & 2ab+\frac{1}{2}\langle w,w\rangle & w\trsp\\
\hline
0 & a-b & 1	& -a+b	& 0 \\
-a-b & -2ab+\frac{1}{2}\langle w,w\rangle	& a-b & 1+2ab-\frac{1}{2}\langle w,w\rangle & w\trsp\\
\hline
0 & w & 0 & -w	& I_{n-2}
\end{array}\right).
\]
In this case we get $N\cdot p \subset \ads \cap (\bW+p)$ where $\bW=\{ x \in \bR^{2,n} : x_2=x_4 \}$. To see the reverse inclusion, let $q \in \ads \cap (\bW+p)$. Then, taking $X\in\g{n}$ as in~\eqref{gen-X-in-n} with
\begin{align*}
a&{}=\frac{(q_1+q_3)-(p_1+p_3)}{2(p_2-p_4)},& b&{}=\frac{(q_1-q_3)-(p_1-p_3)}{2(p_2-p_4)}, &
w=\frac{q-p}{p_2-p_4},
\end{align*}
we get $\Exp(X)\cdot p=q$. Therefore, we have shown

\begin{proposition}\label{N-action}
The nilpotent factor $N$ in the Iwasawa decomposition $SO^0(2,n)=KAN$ acts with cohomogeneity one on $\ads$. For every $p \in \ads$ with $p_2 \neq p_4$, the orbit $N\cdot p$ is principal and obtained as the intersection of $\ads$ with the affine hyperplane $p+\bW$, where $\bW=\{ x \in \bR^{2,n} : x_2=x_4 \}$. If $p_2 = p_4$, then $N\cdot p$ is a singular orbit of dimension $n-1$ that is obtained by intersecting $\ads$ with the affine subspace $p+\bV$, where $\bV=\{ x \in \bR^{2,n} : x_1=x_3, x_2=x_4 \}$.

In particular, principal orbits are parametrized by non-zero values of $p_2-p_4$ and singular orbits (where $p_2=p_4$) are parametrized by non-zero values of $p_1-p_3$.
\end{proposition}

\begin{remark}\label{rm:contain-H}
Let $H$ be a subgroup of $SO^0(2,n)$ such that $N\subset H\subset AN$. We determine when $H$ acts on $\ads$ with cohomogeneity one.

By Proposition~\ref{N-action}, the nilpotent subgroup $N$ acts with cohomogeneity one on $\ads$. Hence, $H$ acts with cohomogeneity one or zero. Thus, it is enough to rule out those subgroups that act with cohomogeneity zero. Consider the Lie algebra $\g{h}$ of $H$, which by assumption satisfies $\g{n}\subset\g{h} \subset \g{a}\oplus \g{n}$. Let $p\in\ads$, and $a$, $b\in\bR$. Then, if $p_2\neq p_4$, by~\eqref{N-prin-Tan-Sp} and Figure~\ref{fig:root-spaces}, we have $H_{a,b}\cdot p\in\g{n}\cdot p$ if and only if $b=0$. Thus, either $\g{h}=\g{n}$ or $\g{h}=\bR H_{1,0}\oplus\g{n}$. In particular, $AN$ acts with cohomogeneity zero on $\ads$.

The orbits of $N$ have already been studied in Proposition~\ref{N-action}. Let $\g{h}=\bR H_{1,0}\oplus\g{n}$. We are now going to describe the orbit foliation induced by $H$. For that matter, let us denote by $\bar{O}_r$ the principal $N$-orbit through a point $p$ with $r=p_2 - p_4\neq 0$, and by $O_s$ the singular $N$-orbit of a point $p$ with $p_2=p_4$ and $s=p_1-p_3$.
Then it is not difficult to see that the orbit foliation induced by $H$ can be written as
\begin{equation}\label{eq:H-orbits}
\FM_H = O_+ \cup O_-  \cup \biggl(\bigcup_{r\neq 0 \in \bR} \bar{O}_r\biggr),
\text{ where $O_+ = \bigcup_{s >0}O_s$, and $O_- = \bigcup_{s <0}O_s$.}
\end{equation}
\end{remark}


\subsection{Parabolic subgroups of $SO^0(2,n)$}
\label{subsec: Parab-Subg-SO(2,n)}\hfill

From the point of view of Riemannian geometry, a parabolic subgroup of the isometry group of a Riemannian symmetric space of noncompact type is the normalizer of a point at infinity. This definition does not apply to our case, but parabolic subgroups of real semisimple Lie groups can be determined, up to conjugacy, from the root space decomposition of its Lie algebra. We take this approach in this section, and determine how parabolic subgroups of $SO^0(2,n)$ act on $\ads$.

We follow~\cite{BoJi06}. The conjugacy classes of parabolic Lie subgroups are parametrized by proper subsets of the set of positive simple roots $\Lambda$. Let $\Phi\varsubsetneq\Lambda$ be a subset of the set of positive simple roots, and denote by $\Sigma_\Phi$ the root system generated by $\Phi$. We write $\Sigma^+_\Phi = \Sigma_\Phi \cap \Sigma^+$, and define
\begin{align}\label{eq:parabolic}
\g{l}_\Phi &{}=\g{g}_0 \oplus \biggl( \bigoplus_{\lambda \in \Sigma_\Phi}\g{g}_\lambda \biggr),&
\g{n}_\Phi &{}=\bigoplus_{\lambda \in \Sigma^+ \backslash \Sigma^+_\Phi}\g{g}_\lambda&
\g{q}_\Phi= \g{l}_\Phi \oplus \g{n}_\Phi.
\end{align}
By definition, $\g{q}_\Phi$ is a parabolic subalgebra associated with the subset $\Phi \subset \Lambda$. The decomposition $\g{q}_\Phi= \g{l}_\Phi \oplus \g{n}_\Phi$ is called the Langlands decomposition of $\g{q}_\Phi$. Moreover, $\g{l}_\Phi$ is a reductive Lie algebra, and $\g{n}_\Phi$ is a nilpotent subalgebra of $\g{g}$. For $\Phi=\emptyset$ we have $\g{l}_\emptyset=\g{g}_0$ and $\g{n}_\emptyset=\g{n}$. Hence,~\eqref{eq:parabolic} becomes
\[
\g{q}_\emptyset=\g{g}_0 \oplus \g{n}=\g{k}_0\oplus\g{a}\oplus\g{n},
\]
which is a minimal parabolic subalgebra.

For $\g{so}(2,n)$ there are two more parabolic subalgebras up to conjugacy. We denote them by $\g{q}_i=\g{q}_{\Phi_i}$, $i=1,2$, where $\Phi_i=\Lambda\setminus\{\alpha_i\}$. Simply using~\eqref{eq:parabolic} we get
\begin{align*}
\g{q}_1&{}
=\g{g}_0\oplus(\g{g}_{\alpha_2}\oplus\g{g}_{-\alpha_2})
\oplus(\g{g}_{\alpha_1}\oplus\g{g}_{\alpha_1+\alpha_2}
\oplus\g{g}_{\alpha_1+2\alpha_2})
=\g{q}_\emptyset\oplus\g{g}_{-\alpha_2},\\
\g{q}_2&{}=
\g{g}_0\oplus(\g{g}_{\alpha_1}\oplus\g{g}_{-\alpha_1})
\oplus(\g{g}_{\alpha_2}\oplus\g{g}_{\alpha_1+\alpha_2}
\oplus\g{g}_{\alpha_1+2\alpha_2})
=\g{q}_\emptyset\oplus\g{g}_{-\alpha_1}.
\end{align*}
We denote by $Q_\emptyset$, $Q_1$ and $Q_2$ the connected Lie subgroups of $SO^0(2,n)$ whose Lie algebras are $\g{q}_\emptyset$, $\g{q}_1$ and $\g{q}_2$, respectively.

A parabolic subgroup of the isometry group of a Riemannian symmetric space of noncompact type acts transitively on that space. However, in our setting we have

\begin{proposition}\label{th:parab-action}
The groups $Q_\emptyset$, $Q_1$ and $Q_2$ act isometrically on $\ads$ with cohomogeneity zero. Furthermore, the group $Q_2$ acts transitively on $\ads$. The orbit foliation induced by $Q_\emptyset$, $Q_1$ and $AN$ coincide, and is given by
\begin{equation} \label{qe=q2-orbdec}
\FM=O_+ \cup O_- \cup \bar{O}_+ \cup \bar{O}_-,
\end{equation}
where $O_+$ and $O_-$ are as in~\eqref{eq:H-orbits}, and
$\bar{O}_+ = \bigcup_{r >0}\bar{O}_r$, $\bar{O}_- = \bigcup_{r <0}\bar{O}_r$.
\end{proposition}

\begin{proof}
Since $\g{q}_\emptyset$, $\g{q}_1$ and $\g{q}_2$ contain $\g{a}\oplus\g{n}$, it is clear that $Q_\emptyset$, $Q_1$ and $Q_2$ act with cohomogeneity zero by Remark~\ref{rm:contain-H}.

Let $p\in\ads$. If $p_2\neq p_4$, since $\g{a}\oplus\g{n}$ is contained in $\g{q}_i$,  $i\in\{\emptyset,1,2\}$, it readily follows from Remark~\ref{rm:contain-H} that $\g{q}_i\cdot p$ is $(n+1)$-dimensional. Assume now $p_2=p_4$. We can easily see, using the same methods as in Section~\ref{sec:N}, that $(\g{a}\oplus\g{n})\cdot p=\bV$, where $\bV=\{x\in\bR^{2,n}:\langle x,p\rangle=0, x_2=x_4\}$ (the equality $\langle x,p\rangle=0$ follows simply from the fact that the action is isometric).

If $X\in\g{g}_{-\alpha_1}=\theta\g{g}_{\alpha_1}$, then it follows from the expressions in Figure~\ref{fig:root-spaces} that $(X\cdot p)_2=-(X\cdot p)_4$. Thus, $T_p(Q_2\cdot p)=\g{q}_2\cdot p=T_p\ads$ and we get that the action of $Q_2$ on $\ads$ is transitive. On the other hand, $\g{g}_{-\alpha_2}\cdot p$, $\g{k}_0\cdot p\subset(\g{a}\oplus\g{n})\cdot p$, and hence $\g{q}_1\cdot p=\g{q}_\emptyset\cdot p=(\g{a}\oplus\g{n})\cdot p\varsubsetneq T_p\ads$. As in Remark~\ref{rm:contain-H}, the corresponding orbits through $p\in\ads$ with $p_2=p_4$ are parametrized by $s=p_1-p_3\neq 0$.
\end{proof}




\bibliographystyle{plain}

\end{document}